\documentclass[preprint,12pt]{elsarticle}

\usepackage{amsmath,amssymb,amsthm,mathtools,bm,mathrsfs}
\usepackage{microtype}
\usepackage{enumitem}
\usepackage{booktabs}
\usepackage{array}
\usepackage{hyperref}
\hypersetup{hidelinks}
\usepackage{aliascnt}
\usepackage[nameinlink,capitalise,noabbrev]{cleveref}

\journal{Journal of Functional Analysis}

\numberwithin{equation}{section}

\newtheorem{theorem}{Theorem}[section]

\newaliascnt{proposition}{theorem}
\newtheorem{proposition}[proposition]{Proposition}
\aliascntresetthe{proposition}

\newaliascnt{lemma}{theorem}
\newtheorem{lemma}[lemma]{Lemma}
\aliascntresetthe{lemma}

\newaliascnt{corollary}{theorem}
\newtheorem{corollary}[corollary]{Corollary}
\aliascntresetthe{corollary}

\newaliascnt{claim}{theorem}

\aliascntresetthe{claim}

\theoremstyle{definition}
\newaliascnt{definition}{theorem}

\aliascntresetthe{definition}

\theoremstyle{remark}
\newaliascnt{remark}{theorem}
\newtheorem{remark}[remark]{Remark}
\aliascntresetthe{remark}

\crefname{theorem}{Theorem}{Theorems}
\Crefname{theorem}{Theorem}{Theorems}
\crefname{proposition}{Proposition}{Propositions}
\Crefname{proposition}{Proposition}{Propositions}
\crefname{lemma}{Lemma}{Lemmas}
\Crefname{lemma}{Lemma}{Lemmas}
\crefname{corollary}{Corollary}{Corollaries}
\Crefname{corollary}{Corollary}{Corollaries}
\crefname{remark}{Remark}{Remarks}
\Crefname{remark}{Remark}{Remarks}

\newcommand{\R}{\mathbb R}
\newcommand{\N}{\mathbb N}
\newcommand{\Z}{\mathbb Z}
\newcommand{\C}{\mathbb C}
\newcommand{\F}{\mathcal F}
\newcommand{\Sspace}{\mathcal S}
\newcommand{\Pcal}{\mathbb P}
\newcommand{\cX}{\mathcal X}
\newcommand{\cE}{\mathcal E}

\newcommand{\supp}{\operatorname{supp}}
\newcommand{\diver}{\operatorname{div}}

\newcommand{\sgn}{\operatorname{sgn}}

\newcommand{\dd}{\,\mathrm d}
\newcommand{\trans}[1]{\tau_{#1}}
\newcommand{\lp}{\dot\Delta}
\newcommand{\low}{\dot S}

\newcommand{\norm}[2][]{\left\lVert #2\right\rVert_{#1}}
\newcommand{\abs}[1]{\left\lvert #1\right\rvert}
\newcommand{\brac}[1]{\left(#1\right)}
\newcommand{\ang}[1]{\left\langle #1\right\rangle}

\begin{document}

\begin{frontmatter}

\title{Magnetic norm inflation at the $\ell^1$ Besov endpoint for viscous non-resistive incompressible MHD}

\author{Yingzhe Ban}
\author{Boling Guo}
\author{Maotuo Guo}

\begin{abstract}
We prove magnetic-field norm inflation at the origin for the viscous, non-resistive incompressible magnetohydrodynamic equations on $\R^d$, $d\ge2$, in the endpoint space
\[
   \dot B^{-1}_{\infty,1}(\R^d)\times \dot B^{0}_{\infty,1}(\R^d).
\]
For every sufficiently small $\varepsilon>0$, we construct smooth divergence-free initial data $(u_0,b_0)$, with $b_0$ compactly supported and
\[
  \norm[\dot B^{-1}_{\infty,1}]{u_0}
  +\norm[\dot B^{0}_{\infty,1}]{b_0}<\varepsilon,
\]
such that the corresponding classical solution satisfies
\[
  \norm[\dot B^{0}_{\infty,1}]{b(t_\varepsilon)}>\varepsilon^{-1}
\]
for some $0<t_\varepsilon<\varepsilon$. Consequently, no solution map agreeing with classical solutions can be continuous at the origin in the magnetic endpoint topology. The low--high paraproduct mechanism available when the dyadic summability exponent is greater than one gives no gain at the $\ell^1$ endpoint. We instead combine a resonant four-wave velocity interaction with magnetic stretching to produce uniformly sized contributions on a growing family of dyadic shells. Frequency-localized test functionals extract the endpoint lower bound, while uniform Lagrangian estimates in weighted Fourier $L^1$ spaces and a spatial localization argument control the remainders and remove the auxiliary constant magnetic field.
\end{abstract}


\end{frontmatter}

\section{Introduction and main result}
\label{sec:intro}

\subsection{The equation and the endpoint problem}

We consider the viscous, non-resistive incompressible magnetohydrodynamic system on $\R^d$, $d\ge2$,
\begin{equation}
\label{eq:MHD}
\left\{
\begin{aligned}
  &\partial_tu-\nu\Delta u+(u\cdot\nabla)u+\nabla p=(b\cdot\nabla)b,\\
  &\partial_tb+(u\cdot\nabla)b=(b\cdot\nabla)u,\\
  &\diver u=\diver b=0,\\
  &(u,b)|_{t=0}=(u_0,b_0),
\end{aligned}
\right.
\end{equation}
where $\nu>0$, $u$ is the velocity, $b$ is the magnetic field, and $p$ is the pressure. Only the velocity is diffused. The magnetic field is transported and stretched by the flow, and the classical local theory therefore places the magnetic datum one derivative above the velocity datum; see \cite{CheminMcCormickRobinsonRodrigo2016,FeffermanMcCormickRobinsonRodrigo2014,FeffermanMcCormickRobinsonRodrigo2017,LiTanYin2017,Wan2016}. Continuous dependence and failure of uniform continuity in Sobolev spaces were studied in \cite{LiYinZhu2023}.

At dyadic summability exponent $r=1$, local existence is known for every finite spatial integrability exponent, together with Hadamard continuity in the ranges established in \cite{LuoYeYin2025,Guo2026}. The limiting pair considered here is
\begin{equation}
\label{eq:endpoint-pair}
  \dot B^{-1}_{\infty,1}(\R^d)\times \dot B^{0}_{\infty,1}(\R^d).
\end{equation}
Throughout the paper, ``endpoint'' refers to the $p=\infty$ and dyadic $\ell^1$ endpoint in this one-derivative-gap scale. We ask whether a solution map that agrees with smooth classical solutions can have a magnetic component continuous at the origin from a neighborhood in \eqref{eq:endpoint-pair} into $C([0,T];\dot B^0_{\infty,1})$.

The magnetic component is essential. Setting $b\equiv0$ reduces \eqref{eq:MHD} to Navier--Stokes, so an instability of the full MHD map could be inherited from the velocity equation and need not detect the nondiffusive transport--stretching dynamics. The construction below makes both initial components tend to zero while the magnetic component becomes arbitrarily large in arbitrarily short time.

Frequency-cascade constructions are central in the ill-posedness theory of fluid equations. Bourgain and Pavlovi\'c introduced a high--high-to--low mechanism for Navier--Stokes norm inflation \cite{BourgainPavlovic2008}; related second-iterate and dyadic constructions appear in \cite{CheskidovShvydkoy2010,Germain2008,Wang2015}. For dissipative MHD, magnetic norm inflation was proved in \cite{DaiQingSchonbek2011,CheskidovDai2015}. Those arguments feed a fixed low mode through a direct quadratic interaction and do not address the zero-order $\ell^1$ topology of the nondiffusive magnetic equation.

For non-resistive MHD, Chen, Nie and Ye proved sharp magnetic norm inflation in threshold Sobolev spaces and in corresponding Besov spaces with dyadic summability exponent $r>1$ \cite{ChenNieYe2024}. Their mechanism is a direct low--high paraproduct whose gain uses the gap between $\ell^r$ and absolute summation. That sequence-space gain disappears at $r=1$. Our construction instead concentrates the velocity input near one carrier annulus and distributes comparable magnetic outputs over a growing family of target shells. Thus the endpoint is not obtained by a formal limit from $r>1$. The sensitivity of incompressible flow to the precise borderline topology is also visible in the Navier--Stokes theory around $BMO^{-1}$; see \cite{KochTataru2001,Yoneda2010}.

A separate issue is the base state in a homogeneous space. Although homogeneous Littlewood--Paley blocks annihilate constants, the normalized realization used below excludes nonzero constant fields rather than identifying them with zero. Moreover, instability near a nonzero magnetic equilibrium is dynamically different from discontinuity at the zero state. Background magnetic fields may change the stability problem substantially; see \cite{BardosSulemSulem1988,AbidiZhang2017,CaiLei2018,DengZhang2018,LinXuZhang2015,PanZhouZhu2018,RenWuXiangZhang2014,TanWang2018,XuZhang2015,Zhang2016}. The current-inflation result of Dolce, Knobel and Zillinger \cite{DolceKnobelZillinger2026} concerns a derivative of the magnetic field near a nonzero sheared state, rather than short-time inflation of $b$ itself near $(0,0)$ in \eqref{eq:endpoint-pair}.

\subsection{Main theorem}

A \emph{classical solution} below means a smooth solution of \eqref{eq:MHD} whose spatial slices belong to $H^m(\R^d)$ for every $m\ge0$. For the smooth data constructed here, this agrees with the standard local strong solution.

\begin{theorem}[Endpoint magnetic norm inflation]
\label{thm:main}
There exists $\varepsilon_0>0$, depending only on $d$, $\nu$, and the fixed Littlewood--Paley cutoffs, such that for every $0<\varepsilon<\varepsilon_0$ one can find divergence-free data
\[
   u_0\in\Sspace(\R^d;\R^d),
   \qquad
   b_0\in C_c^\infty(\R^d;\R^d),
\]
and a time $0<t_\varepsilon<\varepsilon$ for which
\begin{equation}
\label{eq:main-small-data}
  \norm[\dot B^{-1}_{\infty,1}]{u_0}
  +\norm[\dot B^{0}_{\infty,1}]{b_0}<\varepsilon.
\end{equation}
The corresponding classical solution exists on $[0,t_\varepsilon]$ and satisfies
\begin{equation}
\label{eq:main-inflation}
  \norm[\dot B^{0}_{\infty,1}]{b(t_\varepsilon)}>\varepsilon^{-1}.
\end{equation}
\end{theorem}

The magnetic data in \cref{thm:main} are compactly supported and converge to zero in $\dot B^0_{\infty,1}$. A constant field is used only in a single-building-block calculation and is replaced, before the final data are formed, by a compactly supported solenoidal plateau. Hence the conclusion is genuinely based at the zero magnetic state.

\paragraph{Relation to the compensated Piola framework.}
The Lagrangian estimate used in \cref{sec:lagrangian} is closely related to the compensated Piola mechanism developed in the companion preprint \cite{Guo2026}. We include a complete self-contained version of the estimate required here. The resonant four-wave construction, the target-shell lower bound, the frequency-localized remainder analysis, and the zero-state localization are specific to the present endpoint norm-inflation argument.

\subsection{Proof architecture}

Let $K=2^k$ be a large carrier frequency and let $\lambda=2^j$ range over
\[
   K^{1/4}\le \lambda\le K^{1/2}.
\]
For each target shell, four real divergence-free packets in the same carrier annulus are arranged so that their quadratic velocity interaction produces a packet near frequency $\lambda e$. The linear magnetic response stays at the carrier scale and is annihilated by the target-shell test functional. Magnetic stretching of the lower-frequency velocity response then creates a term of size $\eta\delta^2$ in the target magnetic shell. Its scale factors satisfy
\[
   \underbrace{K^2}_{\text{two input amplitudes}}
   \underbrace{\lambda^2}_{\text{two output derivatives}}
   \underbrace{K^{-2}\lambda^{-2}}_{\text{two time scales}}
   \simeq1.
\]
Thus each selected shell receives a contribution bounded below independently of $K$ and $\lambda$. Repeating this over $\simeq k$ distinct shells yields the required $\ell^1$ accumulation.

The lower bound is detected by one uniformly bounded linear functional in each target shell. This avoids the stronger and unnecessary task of estimating the sum of the full absolute Littlewood--Paley remainder blocks. In Lagrangian variables, the frozen-in identity exposes the stretching mechanism. A weighted Fourier $L^1$ estimate controls the nonlinear remainder uniformly at the target scale. The key quantity is $\lambda Y$, where $Y$ is the Lagrangian displacement; the unavailable carrier-scale bound $KY$ is never used.

Finally, the building blocks are translated far apart in physical space. This keeps the $L^\infty$-based input norm independent of the number of shells. An a posteriori stability estimate corrects their superposition to an exact solution, and a compactly supported solenoidal plateau replaces the auxiliary constant field.

The paper is organized as follows. \Cref{sec:prelim} introduces the homogeneous spaces, weighted Fourier norms, and frequency-localized test functionals. \Cref{sec:four-wave} proves the scale-independent target-shell response. \Cref{sec:lagrangian} establishes the uniform Lagrangian estimates, the target-scale displacement bound, and the nonlinear single-shell estimate. \Cref{sec:localization} localizes the magnetic background and glues separated building blocks. \Cref{sec:proof-main} selects the parameters and proves \cref{thm:main}. The high-regularity stability estimate is proved in \ref{app:stability}.

\section{Functional setting and frequency-localized test functionals}
\label{sec:prelim}

The letters $C,c>0$ denote constants whose values may change from line to line. Unless a dependence is stated explicitly, they depend only on $d$, $\nu$, and the fixed cutoff functions. Repeated spatial indices are summed. For vector fields $u,v$ and a matrix field $F$, we set
\[
   (u\otimes v)_{i\ell}=u_i v_\ell,
   \qquad
   (\diver F)_i=\partial_\ell F_{i\ell}.
\]
The Leray projector is denoted by $\Pcal$; for $\xi\ne0$ its symbol is
\begin{equation}
\label{eq:Leray}
   \Pcal(\xi)=I-\frac{\xi\otimes\xi}{|\xi|^2}.
\end{equation}

\subsection{Littlewood--Paley decomposition}

Choose radial functions $\chi,\varphi\in C_c^\infty(\R^d)$ such that
\[
  \supp\chi\subset\{|\xi|\le4/3\},
  \qquad
  \supp\varphi\subset\{3/4\le|\xi|\le8/3\},
\]
$\varphi=1$ on $\{4/5\le|\xi|\le5/4\}$, and
\[
  \chi(\xi)+\sum_{j\ge0}\varphi(2^{-j}\xi)=1,
  \qquad
  \sum_{j\in\Z}\varphi(2^{-j}\xi)=1\quad(\xi\ne0).
\]
Set
\[
   \lp_j f=\F^{-1}\!\left[\varphi(2^{-j}\cdot)\widehat f\right],
   \qquad
   \low_j f=\sum_{\ell<j}\lp_\ell f.
\]
Homogeneous spaces are understood in the normalized realization
\begin{equation}
\label{eq:Sh-prime}
   \Sspace'_h
   :=\left\{f\in\Sspace'(\R^d):\low_jf\to0\text{ in }\Sspace'\text{ as }j\to-\infty\right\},
\end{equation}
so distributions are not taken modulo polynomials. For $s\in\R$,
\begin{equation}
\label{eq:hom-Besov}
   \norm[\dot B^s_{\infty,1}]{f}
   :=\sum_{j\in\Z}2^{js}\norm[L^\infty]{\lp_jf}.
\end{equation}
We use Bony's decomposition and the standard support and multiplier properties of the dyadic blocks; see \cite{Bony1981,BahouriCheminDanchin2011}.

For the smooth stability argument we also use the inhomogeneous blocks $\Delta_{-1}=\chi(D)$ and $\Delta_j=\lp_j$ for $j\ge0$, with
\begin{equation}
\label{eq:inh-Besov}
  \norm[B^s_{\infty,1}]{f}
  :=\sum_{j\ge-1}2^{js}\norm[L^\infty]{\Delta_jf}.
\end{equation}
For a time interval $I$,
\begin{equation}
\label{eq:CL}
  \norm[\widetilde L^\rho(I;B^s_{\infty,1})]{f}
  :=\sum_{j\ge-1}2^{js}\norm[L^\rho(I;L^\infty)]{\Delta_jf}.
\end{equation}
Only noncritical indices $s>1$ are used in this auxiliary argument.

\subsection{Weighted Fourier $L^1$ spaces}

We use the convention
\begin{equation}
\label{eq:Fourier-convention}
  \widehat f(\xi)=\int_{\R^d}e^{-ix\cdot\xi}f(x)\dd x,
  \qquad
  f(x)=(2\pi)^{-d}\int_{\R^d}e^{ix\cdot\xi}\widehat f(\xi)\dd\xi.
\end{equation}
For $\sigma\in\R$, define
\begin{equation}
\label{eq:Xsigma}
   \norm[\cX^\sigma]{f}:=\int_{\R^d}|\xi|^\sigma\abs{\widehat f(\xi)}\dd\xi.
\end{equation}
Vector and matrix norms are taken componentwise. The space $\cX^0$ is a Banach algebra,
\begin{equation}
\label{eq:X0-algebra}
   \norm[\cX^0]{fg}\le \norm[\cX^0]{f}\norm[\cX^0]{g},
\end{equation}
and Fourier inversion gives $\norm[L^\infty]{f}\le C\norm[\cX^0]{f}$. We shall also use
\begin{equation}
\label{eq:X-interpolation}
   \norm[\cX^0]{f}^2\le \norm[\cX^{-1}]{f}\norm[\cX^1]{f}.
\end{equation}
For $T>0$, put
\begin{equation}
\label{eq:ET}
   \norm[\cE_T]{v}
   :=\norm[L^\infty(0,T;\cX^{-1})]{v}
     +\nu\norm[L^1(0,T;\cX^1)]{v}.
\end{equation}
The heat semigroup and the Leray projector satisfy
\begin{align}
\label{eq:heat-X}
  \norm[\cE_T]{e^{\nu t\Delta}f}&\le C\norm[\cX^{-1}]{f},\\
\label{eq:bilinear-X}
  \norm[\cE_T]{\int_0^t e^{\nu(t-s)\Delta}\Pcal\diver(F\otimes G)(s)\dd s}
  &\le C_\nu\norm[\cE_T]{F}\norm[\cE_T]{G}.
\end{align}
Indeed, the divergence cancels the $|\xi|^{-1}$ weight, while \eqref{eq:X-interpolation} and Cauchy--Schwarz in time give
\[
  \int_0^T\norm[\cX^0]{F}\norm[\cX^0]{G}\dd t
  \le C_\nu\norm[\cE_T]{F}\norm[\cE_T]{G}.
\]

\subsection{Frequency-localized test functionals}

The endpoint lower bound will be obtained from one scalar test functional in each target shell.

\begin{lemma}[A bounded target-shell functional]
\label{lem:test-functional}
Let $\lambda=2^j\ge16$, $v\in\mathbb S^{d-1}$, and let $r\in\C^d$ be a unit vector. There is a linear functional $\Lambda_{\lambda,v,r}$ represented by a Schwartz kernel $\psi_{\lambda,v,r}$ such that
\begin{align}
\label{eq:test-pairing}
  \Lambda_{\lambda,v,r}(f)
  &=\int_{\R^d}f(x)\cdot\overline{\psi_{\lambda,v,r}(x)}\dd x,\\
\label{eq:test-bounded}
  \abs{\Lambda_{\lambda,v,r}(f)}
  &\le C\norm[L^\infty]{\lp_jf}.
\end{align}
Moreover, $\widehat\psi_{\lambda,v,r}$ may be supported in a fixed-radius ball centered at $\lambda v$, and there exist constants $C_0,C_1>1$, independent of $\lambda,v,r$, such that for every integer $m\ge0$,
\begin{equation}
\label{eq:test-derivatives}
  \norm[L^1]{D^m\psi_{\lambda,v,r}}
  \le C_0(C_1\lambda)^m.
\end{equation}
Here $\norm[L^1]{D^m\psi}$ denotes the sum of the $L^1$ norms of all derivatives of order $m$.
\end{lemma}

\begin{proof}
Choose $\theta\in C_c^\infty(B(0,1))$ and reduce its support so that the multiplier of $\lp_j$ equals one on $\lambda v+\supp\theta$. Define
\[
  \widehat\psi_{\lambda,v,r}(\xi)=c_\theta\theta(\xi-\lambda v)r,
\]
where $c_\theta$ is fixed. Then
\[
   \psi_{\lambda,v,r}(x)
   =c_\theta e^{i\lambda v\cdot x}\check\theta(x)r.
\]
The $L^1$ norm of the kernel is uniform, and \eqref{eq:test-bounded} follows because $\lp_jf=f$ on the Fourier support of the functional. For a multi-index $\alpha$ with $|\alpha|=m$, Leibniz' rule gives
\[
  \partial^\alpha\psi_{\lambda,v,r}
   =c_\theta e^{i\lambda v\cdot x}
     \sum_{\beta\le\alpha}\binom{\alpha}{\beta}
       (i\lambda v)^{\alpha-\beta}\partial^\beta\check\theta(x)r.
\]
The fixed Schwartz seminorms of $\check\theta$, the number of multi-indices, and the binomial coefficients are bounded by $C_0C_1^m$. Since $\lambda\ge16$, this yields \eqref{eq:test-derivatives}.
\end{proof}

\begin{corollary}[Uniform Taylor bound]
\label{cor:test-Taylor}
Let $\psi_{\lambda,v,r}$ be as in \cref{lem:test-functional}. For every bounded vector field $Y$,
\begin{equation}
\label{eq:Taylor-kernel}
  \sum_{m=0}^\infty\frac1{m!}
  \norm[L^1]{D^m\psi_{\lambda,v,r}[Y^{\otimes m}]}
  \le C_0\exp\!\brac{C_1\lambda\norm[L^\infty]{Y}}.
\end{equation}
In particular, the Taylor series of $\psi_{\lambda,v,r}(\cdot+Y)$ is absolutely summable in $L^1$ whenever $\lambda\norm[L^\infty]{Y}$ is uniformly bounded.
\end{corollary}

\section{A resonant four-wave interaction}
\label{sec:four-wave}

We now construct a velocity building block whose second iterate produces a scale-independent magnetic response in a prescribed target shell. Fix the vectors, embedded in the first two coordinates,
\begin{equation}
\label{eq:fixed-vectors}
\begin{aligned}
  \omega&=2^{-1/2}(1,1,0,\dots,0),
  &a_*&=2^{-1/2}(1,-1,0,\dots,0),\\
  e&=5^{-1/2}(1,2,0,\dots,0).
\end{aligned}
\end{equation}
Then $a_*\cdot\omega=0$, $a_*\cdot e\ne0$, and $\Pcal(e)a_*\ne0$. Choose an even nonnegative function $\rho\in C_c^\infty(B(0,r_0))$, $\rho\not\equiv0$, where $r_0>0$ will be fixed sufficiently small.

Let $k\in16\N$ be large, set $K=2^k$, and define
\begin{equation}
\label{eq:Jk}
  J_k:=\{j\in4\N:k/4\le j\le k/2\},
  \qquad
  \lambda=\lambda_j:=2^j.
\end{equation}
Thus $K^{1/4}\le\lambda\le K^{1/2}$ and $|J_k|\simeq k$. Put
\begin{equation}
\label{eq:packet-centers}
  \xi_c=K\omega,
  \qquad
  \beta_\lambda=\frac{\lambda}{2}e,
  \qquad
  \xi_1=\xi_c+\beta_\lambda,
  \qquad
  \xi_2=-\xi_c+\beta_\lambda.
\end{equation}
Define the real divergence-free Schwartz field $W_{K,\lambda}$ by
\begin{equation}
\label{eq:W-packet}
\begin{aligned}
  \widehat W_{K,\lambda}(\xi)=K\big[&i\Pcal(\xi)a_*\rho(\xi-\xi_1)
  +\Pcal(\xi)a_*\rho(\xi-\xi_2)\\
  &-i\Pcal(\xi)a_*\rho(\xi+\xi_1)
  +\Pcal(\xi)a_*\rho(\xi+\xi_2)\big].
\end{aligned}
\end{equation}
Hermitian symmetry makes $W_{K,\lambda}$ real-valued.

\begin{lemma}[Uniform input bounds]
\label{lem:input-bounds}
For all sufficiently large $k$ and every $j\in J_k$,
\begin{equation}
\label{eq:W-support}
  \supp\widehat W_{K,\lambda}\subset\{\xi:K/2\le|\xi|\le2K\},
\end{equation}
and
\begin{equation}
\label{eq:W-critical}
  \norm[\cX^{-1}]{W_{K,\lambda}}
  +\norm[\dot B^{-1}_{\infty,1}]{W_{K,\lambda}}\le C.
\end{equation}
For every pair of integers $m,N\ge0$,
\begin{equation}
\label{eq:W-decay}
  \abs{\nabla^mW_{K,\lambda}(x)}
  \le C_{m,N}K^{m+1}\ang{x}^{-N}.
\end{equation}
All constants are independent of $k$ and $j$.
\end{lemma}

\begin{proof}
Since $\lambda\le K^{1/2}$, the four Fourier balls in \eqref{eq:W-packet} lie in the annulus \eqref{eq:W-support} for large $K$. Their volumes are fixed and the Fourier amplitude is $O(K)$, hence $\norm[\cX^{-1}]{W_{K,\lambda}}\le C$. Only finitely many dyadic blocks near level $k$ occur and $\norm[L^\infty]{W_{K,\lambda}}\le C K$, which gives the Besov bound. Repeated integration by parts in Fourier variables yields \eqref{eq:W-decay}.
\end{proof}

Set
\begin{align}
\label{eq:V1}
  V^{(1)}(t)&=e^{\nu t\Delta}W_{K,\lambda},\\
\label{eq:V2}
  V^{(2)}(t)&=-\int_0^t e^{\nu(t-s)\Delta}\Pcal\diver\big(V^{(1)}\otimes V^{(1)}\big)(s)\dd s,\\
\label{eq:H1}
  H^{(1)}(t)&=\int_0^t(e\cdot\nabla)V^{(1)}(s)\dd s,\\
\label{eq:H2}
  H^{(2)}(t)&=\int_0^t\Big[(e\cdot\nabla)V^{(2)}
  +(H^{(1)}\cdot\nabla)V^{(1)}-(V^{(1)}\cdot\nabla)H^{(1)}\Big](s)\dd s.
\end{align}
Finally, let
\begin{equation}
\label{eq:TK}
  T_K=c_*K^{-1/2},
\end{equation}
where $c_*>0$ will be chosen sufficiently large, independently of $K$ and $\lambda$.

\subsection{Interaction geometry and the time kernel}

Let
\[
  \gamma_1=\xi_1,
  \quad \gamma_2=\xi_2,
  \quad \gamma_3=-\xi_1,
  \quad \gamma_4=-\xi_2.
\]

\begin{lemma}[The unique interaction producing the positive target]
\label{lem:interaction-geometry}
Fix $r_0>0$ sufficiently small. For all sufficiently large $K$, the only ordered pairs $(a,b)\in\{1,2,3,4\}^2$ for which
\[
  \big(B(\gamma_a,r_0)+B(\gamma_b,r_0)\big)
  \cap B(\lambda e,4r_0)\ne\varnothing
\]
are $(a,b)=(1,2)$ and $(2,1)$. Their center sum is $\gamma_1+\gamma_2=\lambda e$.
\end{lemma}

\begin{proof}
The relevant center sums are
\[
\begin{array}{c|c}
\text{ordered pair type}&\text{center sum}\\ \hline
(1,2),(2,1)&\lambda e\\
(3,4),(4,3)&-\lambda e\\
(1,3),(3,1),(2,4),(4,2)&0\\
(1,4),(4,1)&2\xi_c\\
(2,3),(3,2)&-2\xi_c.
\end{array}
\]
The diagonal sums $2\gamma_a$ have magnitude comparable to $K$. Because $\lambda\le K^{1/2}$, all sums except $\lambda e$ remain a distance comparable to either $\lambda$ or $K$ from $B(\lambda e,4r_0)$ once $K$ is large. The radius of a Minkowski sum of two packet supports is at most $2r_0$, which proves the claim.
\end{proof}

For $T>0$, $\xi\in\R^d$, and $\alpha>0$, define
\begin{equation}
\label{eq:Theta-def}
  \Theta_T(\xi,\alpha)
  :=\int_0^T\int_0^t e^{-\nu(t-s)|\xi|^2}e^{-\nu s\alpha}\dd s\dd t.
\end{equation}
Direct integration gives
\begin{equation}
\label{eq:Theta-formula}
  \Theta_T(\xi,\alpha)
  =\frac1{\nu^2(\alpha-|\xi|^2)}
  \left(\frac{1-e^{-\nu T|\xi|^2}}{|\xi|^2}
  -\frac{1-e^{-\nu T\alpha}}{\alpha}\right).
\end{equation}

\begin{lemma}[Uniform two-scale time-kernel bound]
\label{lem:time-kernel}
There are $c,C>0$ such that the following holds. Choose $c_*$ sufficiently large. Then, for all sufficiently large $K$, whenever
\[
  K^{1/4}\le\lambda\le K^{1/2},
  \qquad |\xi-\lambda e|\le4r_0,
  \qquad \alpha\simeq K^2,
\]
one has
\begin{equation}
\label{eq:Theta-bound}
  \frac{c}{K^2\lambda^2}
  \le \Theta_{T_K}(\xi,\alpha)
  \le \frac{C}{K^2\lambda^2}.
\end{equation}
\end{lemma}

\begin{proof}
On the stated set, $|\xi|\simeq\lambda$, $\alpha-|\xi|^2\simeq K^2$, and
\[
  T_K|\xi|^2\gtrsim c_*K^{-1/2}\lambda^2\ge c_*.
\]
Choose $c_*$ so large that $1-e^{-\nu T_K|\xi|^2}\ge c_\nu>0$. Also $T_K\alpha\simeq K^{3/2}$, so the second exponential in \eqref{eq:Theta-formula} is negligible. Since $\alpha^{-1}\ll\lambda^{-2}$, the expression in parentheses in \eqref{eq:Theta-formula} is comparable to $\lambda^{-2}$. This proves \eqref{eq:Theta-bound}.
\end{proof}

\subsection{Scale-independent target-shell output}

\begin{proposition}[Uniform cubic response in one target shell]
\label{prop:cubic-output}
There exist constants $c_0,C_0>0$ such that, for every sufficiently large $k$ and every $j\in J_k$, there is a functional $\Lambda_{K,\lambda}$ of the form in \cref{lem:test-functional} satisfying
\begin{align}
\label{eq:H1-annihilate}
  \Lambda_{K,\lambda}\big(H^{(1)}(T_K)\big)&=0,\\
\label{eq:H2-lower}
  \abs{\Lambda_{K,\lambda}\big(H^{(2)}(T_K)\big)}&\ge c_0,
\end{align}
and
\begin{equation}
\label{eq:Lambda-shell-bound}
  \abs{\Lambda_{K,\lambda}(f)}
  \le C_0\norm[L^\infty]{\lp_jf}.
\end{equation}
\end{proposition}

\begin{proof}
Write $a_\perp(\xi)=\Pcal(\xi)a_*$ and isolate
\[
   \mathcal T_{K,\lambda}
   :=\int_0^{T_K}(e\cdot\nabla)V^{(2)}(t)\dd t.
\]
By \cref{lem:interaction-geometry}, only the ordered pairs $(\xi_1,\xi_2)$ and $(\xi_2,\xi_1)$ contribute to a fixed-radius ball around $\lambda e$. For $\xi$ in that ball, Fourier transformation of \eqref{eq:V2} gives
\begin{equation}
\label{eq:T-hat}
 \widehat{\mathcal T}_{K,\lambda}(\xi)
 =(2\pi)^{-d}iK^2(e\cdot\xi)\Pcal(\xi)
 \int_{\R^d}\Theta_{T_K}(\xi,\alpha_{\xi,\mu})\mathcal S_\xi(\mu)\dd\mu,
\end{equation}
where
\begin{equation}
\label{eq:alpha}
  \alpha_{\xi,\mu}=|\mu|^2+|\xi-\mu|^2
\end{equation}
and
\begin{align}
\label{eq:S-symbol}
  \mathcal S_\xi(\mu)
  ={}&(\xi\cdot a_\perp(\mu))a_\perp(\xi-\mu)
       \rho(\mu-\xi_1)\rho(\xi-\mu-\xi_2)\\
  &+(\xi\cdot a_\perp(\mu))a_\perp(\xi-\mu)
       \rho(\mu-\xi_2)\rho(\xi-\mu-\xi_1).\nonumber
\end{align}
The sign in \eqref{eq:T-hat} can be read directly from the phases. For either selected ordered pair, the product of packet phases is $i$. The minus sign in \eqref{eq:V2} and the factor $i$ from $\diver$ turn this into a positive real coefficient, while the final derivative $e\cdot\nabla$ contributes the displayed factor $i$. Thus the two ordered pairs have the same phase and cannot cancel.

On the support in \eqref{eq:T-hat}, $\alpha_{\xi,\mu}\simeq K^2$. Uniformly in $j\in J_k$,
\begin{align}
\label{eq:symbol-expansion-1}
  a_\perp(\mu)&=a_*+O(\lambda/K),
  &a_\perp(\xi-\mu)&=a_*+O(\lambda/K),\\
\label{eq:symbol-expansion-2}
  e\cdot\xi&=\lambda+O(r_0),
  &\xi\cdot a_\perp(\mu)&=\lambda(a_*\cdot e)+O(\lambda^2/K+r_0),\\
\label{eq:symbol-expansion-3}
  \Pcal(\xi)a_*&=\Pcal(e)a_*+O(r_0/\lambda).
\end{align}
The first line follows from $a_*\cdot\omega=0$ and $|\mu\mp\xi_c|=O(\lambda)$. By \cref{lem:time-kernel}, the leading vector in the integrand is
\begin{equation}
\label{eq:leading-vector}
  iK^2\lambda^2(a_*\cdot e)\Pcal(e)a_*
  \Theta_{T_K}(\xi,\alpha_{\xi,\mu}),
\end{equation}
with relative error bounded by $C(\lambda/K+r_0/\lambda)$. The two bump products in \eqref{eq:S-symbol} are nonnegative.

Choose $0<r_1<r_0$ so that $\rho*\rho$ is strictly positive on $B(0,r_1)$, and take a nonnegative $\vartheta\in C_c^\infty(B(0,r_1))$ with nonzero integral. Put
\begin{equation}
\label{eq:rout}
   r_{\mathrm{out}}
   =i\,\sgn(a_*\cdot e)\frac{\Pcal(e)a_*}{|\Pcal(e)a_*|}
\end{equation}
and define
\begin{equation}
\label{eq:Lambda-def}
  \Lambda_{K,\lambda}(f)
  :=\Re\int_{\R^d}\vartheta(\xi-\lambda e)
       \widehat f(\xi)\cdot\overline{r_{\mathrm{out}}}\dd\xi.
\end{equation}
This is a functional of the form in \cref{lem:test-functional}. By \eqref{eq:Theta-bound}--\eqref{eq:symbol-expansion-3}, after first fixing $r_0$ and then taking $K$ large,
\begin{equation}
\label{eq:T-lower}
   \Lambda_{K,\lambda}(\mathcal T_{K,\lambda})\ge2c_0
\end{equation}
with $c_0$ independent of $K$ and $\lambda$.

It remains to estimate the Lie-bracket term in \eqref{eq:H2}. Let $\Pi_\lambda$ be a smooth Fourier cutoff to the target ball. Since $V^{(1)}$ stays in the carrier annulus,
\begin{equation}
\label{eq:V1-X0}
   \norm[\cX^0]{V^{(1)}(t)}\le CKe^{-c\nu K^2t},
\end{equation}
and
\begin{equation}
\label{eq:H1-X0}
   \norm[\cX^0]{H^{(1)}(t)}
   \le C\int_0^tK^2e^{-c\nu K^2s}\dd s\le C.
\end{equation}
Because the fields are divergence free,
\[
  (H^{(1)}\cdot\nabla)V^{(1)}-(V^{(1)}\cdot\nabla)H^{(1)}
  =\diver\big(H^{(1)}\otimes V^{(1)}-V^{(1)}\otimes H^{(1)}\big).
\]
After localization, the derivative is measured at the output frequency $\lambda$, and therefore
\begin{align}
\label{eq:bracket-small}
 &\norm[\cX^0]{\Pi_\lambda\int_0^{T_K}
  \big[(H^{(1)}\cdot\nabla)V^{(1)}-(V^{(1)}\cdot\nabla)H^{(1)}\big]\dd t}\\
 &\qquad\le C\lambda\int_0^\infty
   \norm[\cX^0]{H^{(1)}(t)}\norm[\cX^0]{V^{(1)}(t)}\dd t
   \le C\frac{\lambda}{K}.\nonumber
\end{align}
Since $\lambda/K\le K^{-1/2}$, this is smaller than $c_0$ for large $K$, proving \eqref{eq:H2-lower}. Finally, $H^{(1)}$ is supported in the carrier annulus, while the functional is supported near $\lambda e$ with $\lambda\le K^{1/2}$; hence \eqref{eq:H1-annihilate}. The bound \eqref{eq:Lambda-shell-bound} follows from \cref{lem:test-functional}.
\end{proof}

\begin{remark}[Scale balance]
The principal output has schematic size
\[
  K^2\lambda^2(K^{-2}\lambda^{-2})\simeq1.
\]
This scale independence is the source of the later $\ell^1$ accumulation over $|J_k|\simeq k$ target shells.
\end{remark}

\section{Uniform Lagrangian control at the target scale}
\label{sec:lagrangian}

Fix $k$ and $j\in J_k$, abbreviate $W=W_{K,\lambda}$, and use the auxiliary decomposition
\[
  b=\eta e+h.
\]
Here $\delta$ is the amplitude of the velocity packet and $\eta$ is the amplitude of the constant magnetic field. The perturbation $(u,h)$ solves
\begin{equation}
\label{eq:auxiliary-system}
\left\{
\begin{aligned}
  &\partial_tu-\nu\Delta u+(u\cdot\nabla)u+\nabla p
     =\eta(e\cdot\nabla)h+(h\cdot\nabla)h,\\
  &\partial_th+(u\cdot\nabla)h
     =\eta(e\cdot\nabla)u+(h\cdot\nabla)u,\\
  &\diver u=\diver h=0,\\
  &(u,h)|_{t=0}=(\delta W,0).
\end{aligned}
\right.
\end{equation}
The constant field will be replaced by a compactly supported solenoidal plateau in \cref{sec:localization}.

\subsection{Lagrangian formulation}

Let $\Phi(t,a)=a+Y(t,a)$ be the flow of $u$,
\begin{equation}
\label{eq:flow}
  \partial_t\Phi(t,a)=u(t,\Phi(t,a)),
  \qquad
  \Phi(0,a)=a.
\end{equation}
Set
\begin{equation}
\label{eq:Lagrangian-variables}
   F=I+D_aY,
   \qquad A=F^{-1},
   \qquad v=u\circ\Phi,
   \qquad Q=p\circ\Phi.
\end{equation}
Since the flow is volume preserving, the frozen-in formula for the full magnetic field is
\begin{equation}
\label{eq:frozen-in}
  b(t,\Phi(t,a))=F(t,a)b_0(a)=\eta F(t,a)e.
\end{equation}
Consequently,
\begin{equation}
\label{eq:Lagrangian-system}
\left\{
\begin{aligned}
  &\partial_tv-\nu\diver_a(AA^T\nabla_av)+A^T\nabla_aQ
     =\zeta(e\cdot\nabla_a)^2Y,\\
  &\diver_a(Av)=0,\\
  &\partial_tY=v,\\
  &(Y,v)|_{t=0}=(0,\delta W),
\end{aligned}
\right.
\qquad \zeta:=\eta^2.
\end{equation}
We used the Piola identity $\partial_{a_j}A_{ji}=0$ and $(Fe)\cdot\nabla_x=e\cdot\nabla_a$.

\subsection{Compensated Stokes estimates}

At negative Fourier regularity, the generic product estimate $\cX^0\cdot\cX^{-1}\subset\cX^{-1}$ fails. The divergence-free Piola structure supplies the missing output derivative.

\begin{lemma}[Compensated products]
\label{lem:compensated-product}
If $\diver f=0$, then
\begin{align}
\label{eq:comp-prod-1}
  \norm[\cX^{-1}]{(f\cdot\nabla)Z}
  &\le C\norm[\cX^{-1}]{f}\norm[\cX^1]{Z},\\
\label{eq:comp-prod-2}
  \norm[\cX^{-1}]{(f\cdot\nabla)w}
  &\le C\norm[\cX^0]{f}\norm[\cX^0]{w}.
\end{align}
\end{lemma}

\begin{proof}
For \eqref{eq:comp-prod-1}, write the input frequencies as $\lambda_1,\mu$ and the output as $\xi=\lambda_1+\mu$. Since $\lambda_1\cdot\widehat f(\lambda_1)=0$,
\[
  \frac{\abs{\widehat f(\lambda_1)\cdot\mu}}{|\xi|}
  \le C\frac{\abs{\widehat f(\lambda_1)}}{|\lambda_1|}|\mu|.
\]
Indeed, if $|\mu|\ge|\lambda_1|/2$, use $\widehat f(\lambda_1)\cdot\mu=\widehat f(\lambda_1)\cdot\xi$; otherwise $|\xi|\ge|\lambda_1|/2$. Integrating proves \eqref{eq:comp-prod-1}. For \eqref{eq:comp-prod-2}, use $(f\cdot\nabla)w=\diver(w\otimes f)$ and the $\cX^0$ algebra estimate.
\end{proof}

\begin{lemma}[Constant-coefficient Stokes estimate with nonzero divergence]
\label{lem:Stokes-nonzero-div}
Suppose
\begin{equation}
\label{eq:Stokes-nonzero}
  \partial_tz-\nu\Delta z+\nabla\pi=\diver\mathbb K,
  \qquad
  \diver z=g=\diver G,
  \qquad
  \partial_tg=\diver R_g,
\end{equation}
with the natural compatibility condition at $t=0$. Then
\begin{align}
\label{eq:Stokes-estimate}
  \norm[\cE_T]{z}+\norm[L^1(0,T;\cX^{-1})]{\nabla\pi}
  \le C\big(&\norm[\cX^{-1}]{z(0)}
  +\norm[L^1(0,T;\cX^0)]{\mathbb K}
  +\norm[L^\infty(0,T;\cX^{-1})]{G}\\
  &+\nu\norm[L^1(0,T;\cX^0)]{g}
  +\norm[L^1(0,T;\cX^{-1})]{R_g}\big).
\end{align}
\end{lemma}

\begin{proof}
Apply $\Pcal$ to \eqref{eq:Stokes-nonzero}. The solenoidal part obeys the standard heat estimate with forcing $\Pcal\diver\mathbb K$. The gradient part is
\[
  (I-\Pcal)z=\nabla\Delta^{-1}g=\nabla\Delta^{-1}\diver G.
\]
Its $L^\infty_t\cX^{-1}$ norm is bounded by $\norm[L^\infty_t\cX^{-1}]{G}$, while
\[
  \nu\norm[L^1_t\cX^1]{\nabla\Delta^{-1}g}
  \le C\nu\norm[L^1_t\cX^0]{g}.
\]
Applying $I-\Pcal$ to the momentum equation yields
\[
  \nabla\pi=(I-\Pcal)\diver\mathbb K
  -\partial_t\nabla\Delta^{-1}g
  +\nu\Delta\nabla\Delta^{-1}g.
\]
The three terms are controlled by $\mathbb K$, $R_g$, and $\nu g$, respectively. Combining the solenoidal and gradient estimates gives \eqref{eq:Stokes-estimate}.
\end{proof}

\begin{proposition}[Compensated Lagrangian Stokes estimate]
\label{prop:compensated-Stokes}
Let $0<T\le1$, $Y(0)=0$, $\partial_tY=v$, $F=I+DY$, and $A=F^{-1}$. Suppose $\det F=1$ and
\begin{equation}
\label{eq:small-deformation}
  \norm[L^\infty(0,T;\cX^0)]{DY}\le\varepsilon_*
\end{equation}
for a sufficiently small $\varepsilon_*>0$. Let $(v,Q)$ solve
\begin{equation}
\label{eq:variable-Stokes}
\left\{
\begin{aligned}
  &\partial_tv-\nu\diver(AA^T\nabla v)+A^T\nabla Q=\diver H,\\
  &\diver(Av)=0,\\
  &v(0)=v_0,
\end{aligned}
\right.
\end{equation}
with $\diver v_0=0$. Put
\[
  \mathcal N_T:=\norm[\cE_T]{v}+\norm[L^1(0,T;\cX^{-1})]{\nabla Q}.
\]
Then
\begin{equation}
\label{eq:compensated-Stokes-estimate}
  \mathcal N_T
  \le C\left(\norm[\cX^{-1}]{v_0}
  +\norm[L^1(0,T;\cX^0)]{H}\right)
  +C\norm[L^\infty_T\cX^0]{DY}\mathcal N_T
  +C_\nu\mathcal N_T^2.
\end{equation}
The constants are independent of $T$ and of the Fourier supports of $v_0,H$.
\end{proposition}

\begin{proof}
The Neumann series and \eqref{eq:X0-algebra} give
\begin{equation}
\label{eq:A-small}
  \norm[\cX^0]{A-I}+\norm[\cX^0]{AA^T-I}
  \le C\norm[\cX^0]{DY}.
\end{equation}
Define the Piola velocity $\widetilde v=Av$. The constraint and the Piola identity imply $\diver\widetilde v=0$, while
\begin{equation}
\label{eq:Piola-graph}
  v=F\widetilde v=\widetilde v+(\widetilde v\cdot\nabla)Y.
\end{equation}
By \cref{lem:compensated-product} and \eqref{eq:small-deformation}, the map in \eqref{eq:Piola-graph} is coercive on solenoidal fields. In particular,
\begin{align}
\label{eq:tilde-v-Xminus}
  \norm[L^\infty_T\cX^{-1}]{\widetilde v}
  &\le C\norm[L^\infty_T\cX^{-1}]{v},\\
\label{eq:tilde-v-L2}
  \norm[L^2_T\cX^0]{\widetilde v}
  &\le C\norm[L^2_T\cX^0]{v},
  \qquad
  \norm[L^2_T\cX^0]{v}^2\le C_\nu\norm[\cE_T]{v}^2.
\end{align}
Set
\[
   G_v:=v-\widetilde v=(\widetilde v\cdot\nabla)Y,
   \qquad g:=\diver v=\diver G_v.
\]
Then
\begin{equation}
\label{eq:Gv}
   \norm[L^\infty_T\cX^{-1}]{G_v}
   \le C\norm[L^\infty_T\cX^0]{DY}\,\norm[L^\infty_T\cX^{-1}]{v}.
\end{equation}
The Piola identity also yields, componentwise,
\[
  g=(\delta_{ji}-A_{ji})\partial_jv_i,
\]
so
\begin{equation}
\label{eq:g-estimate}
  \nu\norm[L^1_T\cX^0]{g}
  \le C\norm[L^\infty_T\cX^0]{DY}\norm[\cE_T]{v}.
\end{equation}
Differentiating \eqref{eq:Piola-graph} in time gives
\[
  \partial_tv=\partial_t\widetilde v
  +(\partial_t\widetilde v\cdot\nabla)Y
  +(\widetilde v\cdot\nabla)v.
\]
Thus $\partial_tg=\diver R_g$ with
\[
  R_g=(\partial_t\widetilde v\cdot\nabla)Y
      +(\widetilde v\cdot\nabla)v.
\]
Using \cref{lem:compensated-product}, \eqref{eq:tilde-v-L2}, and coercivity of the Piola graph,
\begin{align}
\label{eq:Rg-estimate}
  \norm[L^1_T\cX^{-1}]{R_g}
  &\le C\norm[L^\infty_T\cX^0]{DY}
      \norm[L^1_T\cX^{-1}]{\partial_t\widetilde v}
      +C_\nu\mathcal N_T^2,\\
\label{eq:dt-vtilde}
  \norm[L^1_T\cX^{-1}]{\partial_t\widetilde v}
  &\le C\norm[L^1_T\cX^{-1}]{\partial_tv}+C_\nu\mathcal N_T^2.
\end{align}
Rewrite the momentum equation as
\begin{equation}
\label{eq:constant-rewrite}
  \partial_tv-\nu\Delta v+\nabla Q=\diver\mathbb K,
\end{equation}
where
\begin{equation}
\label{eq:K-tensor}
  \mathbb K=H+\nu(AA^T-I)\nabla v+Q(I-A^T).
\end{equation}
The last term is legitimate because the Piola identity gives
\[
  (I-A^T)\nabla Q=\diver\big(Q(I-A^T)\big).
\]
Moreover, $\norm[\cX^0]{Q}=\norm[\cX^{-1}]{\nabla Q}$, and therefore
\begin{equation}
\label{eq:K-bound}
  \norm[L^1_T\cX^0]{\mathbb K}
  \le \norm[L^1_T\cX^0]{H}
  +C\norm[L^\infty_T\cX^0]{DY}\mathcal N_T.
\end{equation}
Equation \eqref{eq:constant-rewrite} also gives
\begin{equation}
\label{eq:dt-v}
  \norm[L^1_T\cX^{-1}]{\partial_tv}
  \le C\big(\mathcal N_T+\norm[L^1_T\cX^0]{\mathbb K}\big).
\end{equation}
Substituting \eqref{eq:dt-v} and \eqref{eq:K-bound} into \eqref{eq:Rg-estimate}--\eqref{eq:dt-vtilde}, and then using \eqref{eq:Gv}--\eqref{eq:g-estimate} in \cref{lem:Stokes-nonzero-div}, yields \eqref{eq:compensated-Stokes-estimate} after decreasing $\varepsilon_*$.
\end{proof}

\subsection{Uniform analytic solvability}

\begin{proposition}[Uniform Lagrangian majorant]
\label{prop:uniform-Lagrangian}
For every $M_W>0$ there exist constants
\[
   r_*=r_*(d,\nu,M_W)>0,
   \qquad C_*=C_*(d,\nu,M_W)>1,
\]
with the following property. If $W$ is smooth and divergence free, $\norm[\cX^{-1}]{W}\le M_W$, and
\begin{equation}
\label{eq:small-delta-zeta}
  |\delta|+|\zeta|<r_*,
\end{equation}
then \eqref{eq:Lagrangian-system} has a unique smooth solution on $[0,1]$ satisfying
\begin{equation}
\label{eq:uniform-Lagrangian-bound}
  \norm[\cE_1]{v}
  +\norm[L^\infty(0,1;\cX^0)]{DY}
  +\norm[L^1(0,1;\cX^{-1})]{\nabla Q}
  \le C|\delta|.
\end{equation}
The map $(\delta,\zeta)\mapsto(Y,v,Q)$ is real analytic. With
\begin{equation}
\label{eq:analytic-expansion}
  Y=\sum_{r\ge1}\sum_{s\ge0}\delta^r\zeta^sY_{r,s},
  \qquad v=\partial_tY,
\end{equation}
one has, for $N=r+s$,
\begin{equation}
\label{eq:coefficient-majorant}
  \norm[\cE_1]{\partial_tY_{r,s}}
  +\norm[L^\infty(0,1;\cX^0)]{DY_{r,s}}
  +\norm[L^1(0,1;\cX^{-1})]{\nabla Q_{r,s}}
  \le C_*^N.
\end{equation}
\end{proposition}

\begin{proof}
Since $Y(t)=\int_0^tv(s)\dd s$,
\begin{equation}
\label{eq:DY-from-v}
  \norm[L^\infty(0,T;\cX^0)]{DY}
  \le \norm[L^1(0,T;\cX^1)]{v}.
\end{equation}
Write the magnetic forcing as
\[
  \zeta(e\cdot\nabla)^2Y=\diver H_\zeta,
  \qquad
  H_\zeta=\zeta\big((e\cdot\nabla)Y\big)\otimes e.
\]
Then
\begin{equation}
\label{eq:Hzeta}
  \norm[L^1(0,T;\cX^0)]{H_\zeta}
  \le |\zeta|T\norm[L^1(0,T;\cX^1)]{v}.
\end{equation}
For a smooth local solution, \cref{prop:compensated-Stokes}, \eqref{eq:DY-from-v}, and \eqref{eq:Hzeta} imply, as long as the deformation remains small,
\begin{equation}
\label{eq:continuity-ineq}
  \mathcal N_T
  \le C_0|\delta|+C_1\big(|\zeta|+\mathcal N_T\big)\mathcal N_T,
  \qquad 0<T\le1.
\end{equation}
Choose $r_*$ so small that the quadratic terms are absorbable. A standard continuity argument gives $\mathcal N_T\le2C_0|\delta|$ for every $T\le1$. This keeps $F$ invertible and the time-integrated Lipschitz norm small.

For completeness, we indicate the fixed-point construction and parameter dependence. Rewrite \eqref{eq:Lagrangian-system} as the constant-coefficient Stokes system \eqref{eq:constant-rewrite}, with the nonlinear tensor \eqref{eq:K-tensor} and the divergence defect determined by the Piola graph \eqref{eq:Piola-graph}. On the complexification of
\[
  \mathfrak E_T
  :=\left\{(Y,v,\nabla Q):Y(t)=\int_0^tv(s)\dd s,
  \ \norm[\cE_T]{v}+\norm[L^1_T\cX^{-1}]{\nabla Q}<\infty\right\},
\]
the maps
\[
  DY\mapsto(I+DY)^{-1},
  \qquad
  (Y,\widetilde v)\mapsto\widetilde v+(\widetilde v\cdot\nabla)Y,
\]
and all products in \eqref{eq:K-tensor} are analytic on a fixed ball. The estimates in \cref{prop:compensated-Stokes}, applied also to differences, show that the corresponding Stokes solution operator is a contraction on the ball
\[
   \mathbb B_{2C_0|\delta|}\subset\mathfrak E_1
\]
whenever $|\delta|+|\zeta|<r_*$. The contraction constant and the radius of the parameter bidisc are independent of the Fourier support of $W$. Hence the analytic contraction theorem gives analytic dependence on $(\delta,\zeta)$.

The same a priori bound controls the low norms that enter the standard high-regularity Lagrangian parabolic continuation criterion. For each fixed smooth $W$, differentiating \eqref{eq:Lagrangian-system} at an arbitrary Sobolev order gives a linear Gronwall estimate whose coefficients depend on the already controlled deformation and integrated Lipschitz norm. Thus the smooth local solution extends to $[0,1]$.

After complexification, Cauchy's estimate on a parameter bidisc of fixed radius yields \eqref{eq:coefficient-majorant}, after increasing $C_*$. Since the solution is zero when $\delta=0$, every nonzero coefficient has $r\ge1$.
\end{proof}

\subsection{Propagation of Fourier support and the target-scale displacement}

Let
\begin{equation}
\label{eq:Gamma}
  \Gamma_{K,\lambda}
  :=\{\sigma\xi_c+\tau\beta_\lambda:\sigma,\tau\in\{-1,1\}\}.
\end{equation}

\begin{lemma}[Fourier support of the analytic coefficients]
\label{lem:support-propagation}
There is a constant $C_s\ge2$, independent of $K,\lambda,r,s$, such that every coefficient $Y_{r,s}$ of total order $N=r+s$ in \eqref{eq:analytic-expansion} has Fourier support contained in a union of at most $C_s^N$ balls of radius $C_sN$. Each ball is centered at a sum of at most $C_sN$ elements of $\Gamma_{K,\lambda}$. In particular, every center has the form
\begin{equation}
\label{eq:center-form}
  m\xi_c+n\beta_\lambda,
  \qquad m,n\in\Z,
  \qquad |m|+|n|\le C_sN,
\end{equation}
and
\begin{equation}
\label{eq:global-coeff-support}
  \supp\widehat Y_{r,s}\subset B(0,C_sNK).
\end{equation}
\end{lemma}

\begin{proof}
The assertion is proved by induction on $N$. At order one, the only source is $W$, whose support is the union of four fixed-radius balls centered at $\Gamma_{K,\lambda}$. Suppose the statement holds below order $N$. Expand
\[
  A=(I+DY)^{-1}=\sum_{\ell\ge0}(-DY)^\ell
\]
coefficientwise. At total order $N$, only $\ell\le N$ occurs, and every term in the equations for $(Y_{r,s},Q_{r,s})$ is a finite product of lower-order coefficients whose total order is $N$, followed by constant-coefficient Fourier multipliers and heat/Stokes propagators. Products add Fourier supports, whereas the multipliers and propagators preserve them. Thus radii add linearly, the number of balls grows at most exponentially, and every center is a sum of $O(N)$ initial centers. Enlarging $C_s$ closes the induction. Since every element of $\Gamma_{K,\lambda}$ has size $O(K)$, \eqref{eq:global-coeff-support} follows.
\end{proof}

\begin{lemma}[Target-scale displacement]
\label{lem:target-displacement}
There exists $\varepsilon_1>0$ such that, whenever
\begin{equation}
\label{eq:amplitudes-small}
  0<\eta\le\delta\le\varepsilon_1,
\end{equation}
the solution of \eqref{eq:Lagrangian-system} satisfies
\begin{equation}
\label{eq:target-displacement}
  \norm[L^\infty(0,T_K;\cX^0)]{DY}
  +\lambda\norm[L^\infty(0,T_K;\cX^0)]{Y}
  \le C\delta,
\end{equation}
with $C$ independent of $k$ and $j\in J_k$.
\end{lemma}

\begin{proof}
The bound for $DY$ follows from \cref{prop:uniform-Lagrangian}. We prove the estimate for $\lambda Y$ coefficientwise. Choose a sufficiently small geometric constant $\gamma>0$. For $N\le\gamma\lambda$, each ball described in \cref{lem:support-propagation} is either contained in $B(0,C_sN)$ or in $\{|\xi|\ge\lambda/8\}$. Indeed, if the center in \eqref{eq:center-form} has $m=0$ and $n\ne0$, then
\[
  |n\beta_\lambda|\ge\lambda/2,
\]
and the radius is absorbed by choosing $\gamma$ small. If $m\ne0$, then
\[
  |m\xi_c+n\beta_\lambda|
  \ge K-C_sN\lambda/2
  \ge K-C_s\gamma\lambda^2/2
  \ge K/2,
\]
using $\lambda^2\le K$ and a smaller $\gamma$.

On the noncentral part, $|\xi|\ge\lambda/8$, so \eqref{eq:coefficient-majorant} gives
\begin{equation}
\label{eq:noncentral-Y}
  \lambda\norm[L^\infty(0,T_K;\cX^0)]{Y_{r,s}}
  \le C\norm[L^1(0,T_K;\cX^1)]{\partial_tY_{r,s}}
  \le C C_*^N.
\end{equation}
On the central part, $|\xi|\le C_sN$ and $Y_{r,s}(0)=0$, hence
\begin{equation}
\label{eq:central-Y}
\begin{aligned}
  \lambda\norm[L^\infty(0,T_K;\cX^0)]{Y_{r,s}}
  &\le \lambda T_K C_sN
  \norm[L^\infty(0,T_K;\cX^{-1})]{\partial_tY_{r,s}}\\
  &\le C N C_*^N,
\end{aligned}
\end{equation}
where $\lambda T_K\le c_*$ because $\lambda\le K^{1/2}$.

For $N>\gamma\lambda$, the global support bound \eqref{eq:global-coeff-support} gives
\begin{equation}
\label{eq:high-order-Y}
  \lambda\norm[L^\infty(0,T_K;\cX^0)]{Y_{r,s}}
  \le C\lambda T_K NK C_*^N
  \le C NK C_*^N.
\end{equation}
Summing the coefficients of total order $N$ and using $r\ge1$,
\begin{equation}
\label{eq:tail-sum}
  \sum_{N>\gamma\lambda}
  \lambda\norm[L^\infty(0,T_K;\cX^0)]{Y_N}
  \le C\delta K\sum_{N>\gamma\lambda}N^2
       \big(C_*(\delta+\zeta)\big)^{N-1}.
\end{equation}
After decreasing $\varepsilon_1$, the ratio in the geometric tail is strictly less than one. Since $\lambda\ge K^{1/4}$, the exponential tail in $\lambda$ dominates the factor $K$ uniformly. Summing \eqref{eq:noncentral-Y}, \eqref{eq:central-Y}, and \eqref{eq:tail-sum} proves \eqref{eq:target-displacement}.
\end{proof}

\subsection{Identification of the first two magnetic coefficients}

The next lemma makes explicit the bridge between the analytic Lagrangian expansion and the Eulerian fields in \eqref{eq:H1}--\eqref{eq:H2}.

\begin{lemma}[First and second response coefficients]
\label{lem:coefficient-identification}
For $\eta\ne0$, write $h=\eta\mathcal H(\delta,\zeta)$ with $\zeta=\eta^2$. The map $\mathcal H$ extends analytically to a fixed bidisc around $(\delta,\zeta)=(0,0)$, and
\begin{equation}
\label{eq:first-second-coeff}
  \partial_\delta\mathcal H(0,0)=H^{(1)},
  \qquad
  \frac12\partial_\delta^2\mathcal H(0,0)=H^{(2)}.
\end{equation}
\end{lemma}

\begin{proof}
The frozen-in identity gives, in Eulerian variables,
\[
  h(t,x)=\eta\,DY(t,\Phi^{-1}(t,x))e.
\]
By \cref{prop:uniform-Lagrangian} and the analytic inverse function theorem on the small deformation ball, the quotient $\mathcal H=h/\eta$ is analytic in $(\delta,\zeta)$, including at $\eta=0$.

Substituting $h=\eta\mathcal H$ into \eqref{eq:auxiliary-system} gives
\begin{equation}
\label{eq:u-H-system}
\left\{
\begin{aligned}
  &\partial_tu-\nu\Delta u+\Pcal\diver(u\otimes u)
     =\zeta\Pcal\diver\big(e\otimes\mathcal H+\mathcal H\otimes e
       +\mathcal H\otimes\mathcal H\big),\\
  &\partial_t\mathcal H+(u\cdot\nabla)\mathcal H
     =(e\cdot\nabla)u+(\mathcal H\cdot\nabla)u,\\
  &(u,\mathcal H)|_{t=0}=(\delta W,0).
\end{aligned}
\right.
\end{equation}
At $\zeta=0$, expand
\[
  u=\delta V^{(1)}+\delta^2V^{(2)}+O(\delta^3),
  \qquad
  \mathcal H=\delta\mathcal H^{(1)}+\delta^2\mathcal H^{(2)}+O(\delta^3).
\]
The first two velocity coefficients are exactly \eqref{eq:V1}--\eqref{eq:V2}. Comparing powers of $\delta$ in the second equation of \eqref{eq:u-H-system} gives
\[
  \mathcal H^{(1)}(t)=\int_0^t(e\cdot\nabla)V^{(1)}(s)\dd s=H^{(1)}(t),
\]
and
\[
  \mathcal H^{(2)}(t)=\int_0^t\Big[(e\cdot\nabla)V^{(2)}
  +(H^{(1)}\cdot\nabla)V^{(1)}-(V^{(1)}\cdot\nabla)H^{(1)}\Big](s)\dd s
  =H^{(2)}(t).
\]
This is equivalent to \eqref{eq:first-second-coeff}.
\end{proof}

\subsection{The nonlinear building-block estimate}

\begin{proposition}[Single-shell magnetic lower bound]
\label{prop:single-shell}
There exist $\varepsilon_1,c_1,C_1>0$, independent of $k$ and $j\in J_k$, such that whenever $0<\eta\le\delta\le\varepsilon_1$, problem \eqref{eq:auxiliary-system} has a smooth solution on $[0,T_K]$ and
\begin{equation}
\label{eq:single-shell-lower}
  \abs{\Lambda_{K,\lambda}(h(T_K))}
  \ge c_1\eta\delta^2-C_1(\eta\delta^3+\eta^3).
\end{equation}
Equivalently,
\begin{equation}
\label{eq:single-shell-remainder}
  \abs{\Lambda_{K,\lambda}\big(h(T_K)-\eta\delta H^{(1)}(T_K)
  -\eta\delta^2H^{(2)}(T_K)\big)}
  \le C_1(\eta\delta^3+\eta^3).
\end{equation}
\end{proposition}

\begin{proof}
Let $\psi_{K,\lambda}$ denote the kernel representing $\Lambda_{K,\lambda}$. By \eqref{eq:frozen-in}, volume preservation, and $\widehat\psi_{K,\lambda}(0)=0$,
\begin{align}
\label{eq:Lambda-Lagrangian}
  \Lambda_{K,\lambda}(h(T_K))
  &=\eta\int_{\R^d}\big((I+DY(T_K,a))e-e\big)
     \cdot\overline{\psi_{K,\lambda}(a+Y(T_K,a))}\dd a\\
  &=\eta\int_{\R^d}(I+DY(T_K,a))e
     \cdot\overline{\psi_{K,\lambda}(a+Y(T_K,a))}\dd a.\nonumber
\end{align}
The constant term vanishes after the volume-preserving change of variables $x=\Phi(T_K,a)$.

By \cref{cor:test-Taylor,lem:target-displacement}, the Taylor series in \eqref{eq:Lambda-Lagrangian} is absolutely summable with a majorant independent of $K$ and $\lambda$. Together with \cref{prop:uniform-Lagrangian}, this shows that
\begin{equation}
\label{eq:Phi-analytic}
  \eta^{-1}\Lambda_{K,\lambda}(h(T_K))
  =\Phi_{K,\lambda}(\delta,\zeta),
  \qquad \zeta=\eta^2,
\end{equation}
where $\Phi_{K,\lambda}$ is analytic in a fixed bidisc and has a coefficient majorant independent of $K,\lambda$. Since the solution is stationary when $\delta=0$, $\Phi_{K,\lambda}(0,\zeta)=0$. Cauchy's estimate therefore gives
\begin{equation}
\label{eq:Cauchy-remainder}
 \abs{\Phi_{K,\lambda}(\delta,\zeta)
 -\delta\partial_\delta\Phi_{K,\lambda}(0,0)
 -\frac{\delta^2}{2}\partial_\delta^2\Phi_{K,\lambda}(0,0)}
 \le C(\delta^3+\zeta).
\end{equation}
By \cref{lem:coefficient-identification}, the first two coefficients are the pairings with $H^{(1)}$ and $H^{(2)}$. Applying \cref{prop:cubic-output} to \eqref{eq:Cauchy-remainder} gives
\[
  \abs{\Lambda_{K,\lambda}(h(T_K))}
  \ge c_0\eta\delta^2-C\eta(\delta^3+\eta^2),
\]
which is \eqref{eq:single-shell-lower} after adjusting constants.
\end{proof}

\begin{remark}[Why one directional functional is sufficient]
A lower bound for $\norm[\dot B^0_{\infty,1}]{b}$ requires only one bounded functional in each shell:
\[
  \norm[L^\infty]{\lp_jb}
  \ge C^{-1}\abs{\Lambda_{K,\lambda}(b)}.
\]
No estimate of the sum of the full absolute remainder blocks is needed. Each target shell may be tested at its own spatial point and output direction.
\end{remark}

\section{Localization of the background and spatial decoupling}
\label{sec:localization}

The preceding construction uses a constant magnetic field. We now replace it by a compactly supported solenoidal plateau and place the building blocks far enough apart that all mutual interactions are perturbative.

\subsection{A compact solenoidal plateau}

\begin{lemma}[Localized constant vector]
\label{lem:plateau}
For every $e\in\mathbb S^{d-1}$ there is $G\in C_c^\infty(\R^d;\R^d)$ such that
\begin{equation}
\label{eq:G-plateau}
  \diver G=0,
  \qquad G=e\quad\text{on }B(0,2).
\end{equation}
For a dyadic number $R\ge1$, set $G_R(x)=G(x/R)$. Then
\begin{align}
\label{eq:G-Besov}
  \norm[\dot B^0_{\infty,1}]{G_R}&=\norm[\dot B^0_{\infty,1}]{G},\\
\label{eq:G-derivatives}
  \norm[L^\infty]{\nabla^mG_R}&\le C_mR^{-m},\\
\label{eq:G-high-block}
  \norm[L^\infty]{\lp_jG_R}&\le C_N(2^jR)^{-N}
  \qquad\text{when }2^jR\ge1.
\end{align}
\end{lemma}

\begin{proof}
Choose $\chi_0\in C_c^\infty(\R^d)$ with $\chi_0=1$ on $B(0,2)$ and define the antisymmetric tensor
\[
  M_{i\ell}(x)=\frac{\chi_0(x)}{d-1}(e_ix_\ell-e_\ell x_i).
\]
Set $G_i=\sum_\ell\partial_\ell M_{i\ell}$. Antisymmetry gives $\diver G=0$, and on the region where $\chi_0=1$,
\[
  \sum_\ell\partial_\ell(e_ix_\ell-e_\ell x_i)=(d-1)e_i.
\]
Thus $G=e$ there. The remaining estimates follow from dyadic scaling and rapid Fourier decay.
\end{proof}

\subsection{Sparse translations and far-field products}

For $y\in\R^d$, write $\trans y f(x)=f(x-y)$.

\begin{lemma}[Sparse sum at one carrier scale]
\label{lem:sparse-sum}
For every fixed $k$ and every prescribed separation length $L>0$, the points $y_{k,j}$, $j\in J_k$, may be chosen with pairwise distances at least $L$ so that
\begin{equation}
\label{eq:sparse-weight}
  \sup_{x\in\R^d}\sum_{j\in J_k}\ang{x-y_{k,j}}^{-d-2}\le C
\end{equation}
and
\begin{equation}
\label{eq:sparse-Besov}
  \norm[\dot B^{-1}_{\infty,1}]{\sum_{j\in J_k}\trans{y_{k,j}}W_{K,2^j}}
  \le C,
\end{equation}
where $C$ is independent of $k$ and $L$.
\end{lemma}

\begin{proof}
Let $\ell=\max\{L,1\}$, enumerate $J_k$, choose distinct consecutive integers $r_j$, and set $y_{k,j}=\ell r_je_1$. The lattice sum in \eqref{eq:sparse-weight} is uniformly bounded. The Fourier support of the sum in \eqref{eq:sparse-Besov} remains in the carrier annulus, so only finitely many blocks with $2^n\simeq K$ contribute. By \eqref{eq:W-decay} and \eqref{eq:sparse-weight}, the $L^\infty$ norm of each such block is $O(K)$. Multiplication by $2^{-n}\simeq K^{-1}$ gives \eqref{eq:sparse-Besov}.
\end{proof}

The next lemma replaces the unsupported assertion that nonlinear profiles remain Schwartz at positive times.

\begin{lemma}[Uniform decay of products under translation]
\label{lem:far-field-products}
Let $T>0$ and let $f,g\in C([0,T];H^m(\R^d))$ for every $m\ge0$. Then, for every $r\ge0$,
\begin{equation}
\label{eq:translated-product}
  \sup_{0\le t\le T}\norm[B^r_{\infty,1}]{f(t)\,\trans y g(t)}
  \longrightarrow0
  \qquad\text{as }|y|\to\infty.
\end{equation}
The same statement holds for finite sums of derivatives and tensor products. Moreover, every such trajectory and all of its derivatives vanish uniformly at spatial infinity.
\end{lemma}

\begin{proof}
Fix an integer $M>r+d/2+2$. It is enough to prove convergence in $H^M$, since $H^{M'}\hookrightarrow B^r_{\infty,1}$ for sufficiently large $M'$. Expand derivatives by Leibniz' rule. Every factor derivative belongs to $L^2\cap C_0$, uniformly along the compact trajectories in the relevant Sobolev spaces. Split one factor into a compact part and an $L^2$ tail, uniformly in time. On the compact part, the translated second factor converges uniformly to zero; on the tail, use the uniform $L^\infty$ bound of the second factor. This proves convergence of each differentiated product in $L^2$, uniformly in time. The uniform $C_0$ decay follows by the same compactness argument and Sobolev embedding.
\end{proof}

\subsection{A posteriori stability}

We shall use the following high-regularity stability statement. Its constants may depend on the fixed smooth approximate solution; no uniform high-frequency lifespan is required.

\begin{proposition}[Stability of a smooth approximate solution]
\label{prop:stability}
Let $s>1$ be noninteger and $T>0$. Suppose $(u^a,b^a)$ is a smooth divergence-free pair on $[0,T]\times\R^d$ satisfying the projected equations with residuals $(R_u,R_b)$ and
\begin{equation}
\label{eq:approx-bound}
  \norm[\widetilde L^\infty_TB^{s-1}_{\infty,1}]{u^a}
  +\nu\norm[L^1_TB^{s+1}_{\infty,1}]{u^a}
  +\norm[\widetilde L^\infty_TB^{s+1}_{\infty,1}]{b^a}
  \le M.
\end{equation}
There exist $\rho_*=
ho_*(M,T,\nu)>0$ and $C_a=C_a(M,T,\nu)<\infty$ such that, if
\begin{equation}
\label{eq:residual-small}
  \norm[L^1_TB^{s-1}_{\infty,1}]{R_u}
  +\norm[L^1_TB^{s}_{\infty,1}]{R_b}<\rho_*,
\end{equation}
then the exact solution with initial data $(u^a(0),b^a(0))$ exists on $[0,T]$. With $w=u-u^a$ and $z=b-b^a$,
\begin{align}
\label{eq:stability-bound}
 &\norm[\widetilde L^\infty_TB^{s-1}_{\infty,1}]{w}
  +\nu\norm[L^1_TB^{s+1}_{\infty,1}]{w}
  +\norm[\widetilde L^\infty_TB^{s}_{\infty,1}]{z}\\
 &\qquad\le C_a\left(
  \norm[L^1_TB^{s-1}_{\infty,1}]{R_u}
  +\norm[L^1_TB^{s}_{\infty,1}]{R_b}\right).\nonumber
\end{align}
In particular, the right-hand side controls
\[
  \sup_{0\le t\le T}\big(\norm[L^\infty]{w(t)}+\norm[L^\infty]{z(t)}\big).
\]
\end{proposition}

The proof, including the coupled interval-by-interval bootstrap, is given in \ref{app:stability}.

\subsection{Gluing separated building blocks}

For every $j\in J_k$, let $(u_{k,j},h_{k,j})$ be the solution from \cref{prop:single-shell}, with target $\lambda=2^j$.

\begin{proposition}[Separated-profile gluing]
\label{prop:gluing}
Fix $k$, $0<\eta\le\delta\le\varepsilon_1$, and $\sigma>0$. There are centers $y_{k,j}$ and a dyadic radius $R_k$ such that the exact solution of \eqref{eq:MHD} with data
\begin{equation}
\label{eq:glued-data}
  u_{0,k}=\delta\sum_{j\in J_k}\trans{y_{k,j}}W_{K,2^j},
  \qquad
  b_{0,k}=\eta G_{R_k},
\end{equation}
exists on $[0,T_K]$. If
\begin{equation}
\label{eq:approx-sum}
  u_k^a=\sum_{j\in J_k}\trans{y_{k,j}}u_{k,j},
  \qquad
  b_k^a=\eta G_{R_k}+\sum_{j\in J_k}\trans{y_{k,j}}h_{k,j},
\end{equation}
then
\begin{equation}
\label{eq:gluing-Linf}
  \sup_{0\le t\le T_K}
  \left(\norm[L^\infty]{u_k(t)-u_k^a(t)}
  +\norm[L^\infty]{b_k(t)-b_k^a(t)}\right)<\sigma.
\end{equation}
Moreover,
\begin{equation}
\label{eq:glued-input-bound}
  \norm[\dot B^{-1}_{\infty,1}]{u_{0,k}}
  +\norm[\dot B^0_{\infty,1}]{b_{0,k}}
  \le C(\delta+\eta),
\end{equation}
with $C$ independent of $k$.
\end{proposition}

\begin{proof}
Fix a noninteger $s>2$. For fixed $k$, the family $J_k$ is finite. Write
\[
  U_j=\trans{y_{k,j}}u_{k,j},
  \qquad
  H_j=\trans{y_{k,j}}h_{k,j}.
\]
Every profile trajectory belongs to $C([0,T_K];H^m)$ for all $m$. By translation invariance, the triangle inequality, and \eqref{eq:G-derivatives}, there is a finite $M_k$ such that every pair
\[
  \left(\sum_{j\in J_k}U_j,\ \eta G_R+\sum_{j\in J_k}H_j\right),
  \qquad R\ge1,
\]
satisfies \eqref{eq:approx-bound} with $M=M_k$. Let $\rho_*$ and $C_a$ be the corresponding constants in \cref{prop:stability}, and choose $\mu>0$ so that
\begin{equation}
\label{eq:mu-choice}
  \mu<\rho_*\qquad\text{and}\qquad C_a\mu<\sigma.
\end{equation}

By \cref{lem:far-field-products}, the centers can be chosen recursively, while preserving \cref{lem:sparse-sum}, so far apart that
\begin{align}
\label{eq:cross-small}
 &\sum_{i\ne\ell}\left(
 \norm[L^1(0,T_K;B^s_{\infty,1})]{U_i\otimes U_\ell}
 +\norm[L^1(0,T_K;B^s_{\infty,1})]{H_i\otimes H_\ell}\right)\\
 &\quad+\sum_{i\ne\ell}
 \norm[L^1(0,T_K;B^{s+1}_{\infty,1})]{H_\ell\otimes U_i-U_i\otimes H_\ell}
 <c\mu.\nonumber
\end{align}
Fix those centers. Since $G_R=e$ on $B(0,2R)$, choose $R_k$ larger than every $|y_{k,j}|$. By the uniform spatial decay in \cref{lem:far-field-products}, after increasing $R_k$,
\begin{align}
\label{eq:plateau-tail-small}
 &\eta\sum_{j\in J_k}\norm[L^1(0,T_K;B^s_{\infty,1})]
   {(G_{R_k}-e)\otimes H_j+H_j\otimes(G_{R_k}-e)}\\
 &\quad+\eta\sum_{j\in J_k}\norm[L^1(0,T_K;B^{s+1}_{\infty,1})]
   {(G_{R_k}-e)\otimes U_j-U_j\otimes(G_{R_k}-e)}
 <c\mu.\nonumber
\end{align}
Finally, $\diver G_{R_k}=0$ and at least one derivative falls on the slowly varying factor, so
\begin{equation}
\label{eq:plateau-self}
  \norm[B^{s-1}_{\infty,1}]{\Pcal\diver(G_{R_k}\otimes G_{R_k})}
  \le C_sR_k^{-1}.
\end{equation}
Increase $R_k$ once more so that $\eta^2C_sR_k^{-1}<c\mu$.

The residuals of \eqref{eq:approx-sum} are
\begin{align}
\label{eq:Ru}
  R_{u,k}={}&\Pcal\diver\sum_{i\ne\ell}(U_i\otimes U_\ell-H_i\otimes H_\ell)\nonumber\\
  &-\eta\Pcal\diver\sum_{j\in J_k}\big((G_{R_k}-e)\otimes H_j
      +H_j\otimes(G_{R_k}-e)\big)\nonumber\\
  &-\eta^2\Pcal\diver(G_{R_k}\otimes G_{R_k}),\\
\label{eq:Rb}
  R_{b,k}={}&\diver\sum_{i\ne\ell}(H_\ell\otimes U_i-U_i\otimes H_\ell)\nonumber\\
  &+\eta\diver\sum_{j\in J_k}\big((G_{R_k}-e)\otimes U_j
      -U_j\otimes(G_{R_k}-e)\big).
\end{align}
The bounds \eqref{eq:cross-small}--\eqref{eq:plateau-self} and standard Besov multipliers give
\[
  \norm[L^1(0,T_K;B^{s-1}_{\infty,1})]{R_{u,k}}
  +\norm[L^1(0,T_K;B^s_{\infty,1})]{R_{b,k}}<\mu.
\]
Applying \cref{prop:stability} and \eqref{eq:mu-choice} yields \eqref{eq:gluing-Linf}. The input estimate follows from \cref{lem:sparse-sum,lem:plateau}:
\[
  \norm[\dot B^{-1}_{\infty,1}]{u_{0,k}}\le C\delta,
  \qquad
  \norm[\dot B^0_{\infty,1}]{b_{0,k}}\le C\eta.
\]
\end{proof}

Translate the target-shell functional together with its building block:
\begin{equation}
\label{eq:translated-functional}
  \Lambda^y_{K,\lambda}(f):=\Lambda_{K,\lambda}(\trans{-y}f).
\end{equation}

\begin{corollary}[Decoupling of the target-shell functionals]
\label{cor:functional-decoupling}
Under the hypotheses of \cref{prop:gluing}, the centers and $R_k$ may be chosen so that, for every $j\in J_k$,
\begin{align}
\label{eq:functional-cross}
  \abs{\Lambda^{y_{k,j}}_{K,2^j}(\eta G_{R_k})}
  +\sum_{\ell\ne j}\abs{\Lambda^{y_{k,j}}_{K,2^j}
     (\trans{y_{k,\ell}}h_{k,\ell}(T_K))}&<\sigma,\\
\label{eq:functional-correction}
  \abs{\Lambda^{y_{k,j}}_{K,2^j}\big(b_k(T_K)-b_k^a(T_K)\big)}&<C\sigma.
\end{align}
\end{corollary}

\begin{proof}
The kernel of $\Lambda^y_{K,\lambda}$ is a translated modulation of a fixed Schwartz function. Pairing it with a different translated profile tends to zero as the separation tends to infinity; this uses only the $C_0$ decay supplied by \cref{lem:far-field-products}. Since $J_k$ is finite, all cross pairings may be made small simultaneously. By \eqref{eq:G-high-block} and \eqref{eq:Lambda-shell-bound}, increasing $R_k$ also makes the plateau pairing small. These choices do not affect \eqref{eq:glued-input-bound}. Finally, \eqref{eq:gluing-Linf} and the uniform $L^1$ norm of the test kernel imply \eqref{eq:functional-correction}.
\end{proof}

\section{Proof of endpoint norm inflation}
\label{sec:proof-main}

\begin{proof}[Proof of \cref{thm:main}]
Choose a sufficiently small fixed constant $\kappa>0$. For $n\ge1$, set
\begin{equation}
\label{eq:parameter-choice}
  \delta_n=\kappa n^{-1},
  \qquad
  \eta_n=\kappa n^{-2},
  \qquad
  k_n=16\left\lceil\frac{n^6}{16}\right\rceil,
  \qquad
  t_n=T_{k_n}.
\end{equation}
Thus $k_n\in16\N$, $k_n\simeq n^6$, and $t_n=c_*2^{-k_n/2}\to0$. Apply \cref{prop:gluing} with
\begin{equation}
\label{eq:sigma-n}
  \sigma_n=\frac{\eta_n\delta_n^2}{k_n^2}.
\end{equation}
Let $(u_{0,n},b_{0,n})$ and $(u_n,b_n)$ be the resulting initial data and exact solution. By \eqref{eq:glued-input-bound},
\begin{equation}
\label{eq:input-to-zero}
  \norm[\dot B^{-1}_{\infty,1}]{u_{0,n}}
  +\norm[\dot B^0_{\infty,1}]{b_{0,n}}
  \le C(\delta_n+\eta_n)\longrightarrow0.
\end{equation}

For every $j\in J_{k_n}$, \cref{prop:single-shell,cor:functional-decoupling} give
\begin{align}
\label{eq:each-shell-lower}
  \norm[L^\infty]{\lp_jb_n(t_n)}
  &\ge C^{-1}\abs{\Lambda^{y_{k_n,j}}_{2^{k_n},2^j}(b_n(t_n))}\\
  &\ge c\eta_n\delta_n^2
  -C(\eta_n\delta_n^3+\eta_n^3)-C\sigma_n.\nonumber
\end{align}
Summing over $J_{k_n}$ and using $|J_{k_n}|\simeq k_n$,
\begin{equation}
\label{eq:sum-lower}
  \norm[\dot B^0_{\infty,1}]{b_n(t_n)}
  \ge c\eta_n\delta_n^2k_n
  -C\eta_n\delta_n^3k_n
  -C\eta_n^3k_n
  -Ck_n\sigma_n.
\end{equation}
The four terms on the right have sizes
\begin{equation}
\label{eq:term-sizes}
  c\kappa^3n^2,
  \qquad O(\kappa^4n),
  \qquad O(\kappa^3),
  \qquad O(\kappa^3n^{-10}),
\end{equation}
respectively. For fixed sufficiently small $\kappa$, the first term dominates as $n\to\infty$, and therefore
\begin{equation}
\label{eq:norm-to-infty}
  \norm[\dot B^0_{\infty,1}]{b_n(t_n)}\longrightarrow\infty.
\end{equation}
Given $0<\varepsilon<\varepsilon_0$, choose $n$ so large that the left side of \eqref{eq:input-to-zero} is below $\varepsilon$, $t_n<\varepsilon$, and the norm in \eqref{eq:norm-to-infty} exceeds $\varepsilon^{-1}$. This proves \eqref{eq:main-small-data}--\eqref{eq:main-inflation}.
\end{proof}

\begin{remark}
The proof uses only one directional functional in each selected shell. The total tested error is bounded by
\[
  Ck(\eta\delta^3+\eta^3)+o(\eta\delta^2k),
\]
whereas the leading tested contribution is bounded below by $c\eta\delta^2k$. No upper bound for the sum of the full $L^\infty$ norms of all remainder blocks is required.
\end{remark}

\section{Concluding remarks}

The endpoint mechanism has two essential components: a scale-independent response repeated on a growing family of dyadic shells, and a frequency-localized remainder estimate that preserves the shellwise lower bounds without controlling the full absolute remainder. The compact solenoidal plateau transfers the constant-field building-block calculation to data converging to the zero magnetic state. The argument suggests a possible endpoint strategy for other partially dissipative transport--stretching systems, but the scale balance, frozen-in identity, and compensated Piola structure must be re-established in each model.

A natural further question is whether one can construct a single datum in \eqref{eq:endpoint-pair} whose magnetic field fails to belong to $\dot B^0_{\infty,1}$ at every prescribed positive time. The finite-shell estimate proved here does not provide the persistence or synchronization mechanism required for that stronger pointwise-in-time statement.

\appendix

\section{High-regularity a posteriori stability}
\label{app:stability}

We prove \cref{prop:stability}. Write the projected equations for the approximate pair as
\begin{equation}
\label{eq:approx-system-app}
\left\{
\begin{aligned}
  &\partial_tu^a-\nu\Delta u^a
   +\Pcal\diver(u^a\otimes u^a-b^a\otimes b^a)=R_u,\\
  &\partial_tb^a+\diver(b^a\otimes u^a-u^a\otimes b^a)=R_b,\\
  &\diver u^a=\diver b^a=0.
\end{aligned}
\right.
\end{equation}
Let $w=u-u^a$ and $z=b-b^a$. Then $(w,z)|_{t=0}=0$ and
\begin{equation}
\label{eq:error-system}
\left\{
\begin{aligned}
  &\partial_tw-\nu\Delta w
  +\Pcal\diver(u^a\otimes w+w\otimes u^a+w\otimes w)\\
  &\qquad=\Pcal\diver(b^a\otimes z+z\otimes b^a+z\otimes z)-R_u,\\
  &\partial_tz+(u^a+w)\cdot\nabla z
   =(b^a+z)\cdot\nabla w+z\cdot\nabla u^a-w\cdot\nabla b^a-R_b,\\
  &\diver w=\diver z=0.
\end{aligned}
\right.
\end{equation}
For an interval $I=[t_0,t_1]\subset[0,T]$, put
\begin{align*}
  E_u(I)&:=\norm[\widetilde L^\infty(I;B^{s-1}_{\infty,1})]{w}
  +\nu\norm[L^1(I;B^{s+1}_{\infty,1})]{w},\\
  E_b(I)&:=\norm[\widetilde L^\infty(I;B^s_{\infty,1})]{z},\\
  A_u(I)&:=\norm[L^1(I;B^{s+1}_{\infty,1})]{u^a},\\
  A_b(I)&:=\norm[\widetilde L^\infty(I;B^{s+1}_{\infty,1})]{b^a},\\
  \mathcal R_u(I)&:=\norm[L^1(I;B^{s-1}_{\infty,1})]{R_u},\\
  \mathcal R_b(I)&:=\norm[L^1(I;B^{s}_{\infty,1})]{R_b}.
\end{align*}
The transport--diffusion estimate, Bony's decomposition, and the algebra property of $B^s_{\infty,1}$ give
\begin{align}
\label{eq:Eu-app}
 E_u(I)\le Ce^{CA_u(I)}\big[&\norm[B^{s-1}_{\infty,1}]{w(t_0)}+\mathcal R_u(I)
 +CA_u(I)E_u(I)+CE_u(I)^2\\
 &+C|I|\big(A_b(I)+E_b(I)\big)E_b(I)\big].\nonumber
\end{align}
Indeed, $w\cdot\nabla u^a$ and $w\cdot\nabla w$ are bounded in $L^1(I;B^{s-1}_{\infty,1})$ by $A_u(I)E_u(I)$ and $E_u(I)^2$, while
\[
  \norm[B^{s-1}_{\infty,1}]{b^a\cdot\nabla z+z\cdot\nabla b^a+z\cdot\nabla z}
  \le C\big(A_b(I)+E_b(I)\big)E_b(I).
\]
Let
\[
  V_I=\int_I\norm[B^{s+1}_{\infty,1}]{u^a+w}\dd t
  \le A_u(I)+\nu^{-1}E_u(I).
\]
The transport estimate for the second equation in \eqref{eq:error-system} gives
\begin{align}
\label{eq:Eb-app}
 E_b(I)\le e^{CV_I}\big[&\norm[B^s_{\infty,1}]{z(t_0)}+\mathcal R_b(I)
 +C\nu^{-1}\big(A_b(I)+E_b(I)\big)E_u(I)\\
 &+CA_u(I)E_b(I)+C\nu^{-1}A_b(I)E_u(I)\big].\nonumber
\end{align}
The final term uses
\[
  \norm[B^s_{\infty,1}]{w\cdot\nabla b^a}
  \le C\norm[B^{s+1}_{\infty,1}]{w}\,\norm[B^{s+1}_{\infty,1}]{b^a}.
\]

We now make the coupled bootstrap explicit. Choose a universal $\theta>0$ sufficiently small. By \eqref{eq:approx-bound}, partition $[0,T]$ into finitely many intervals $I_m$ such that on every interval
\begin{equation}
\label{eq:partition-smallness}
  A_u(I_m)\le\theta,
  \qquad
  |I_m|A_b(I_m)\le\theta,
  \qquad
  \frac{|I_m|A_b(I_m)^2}{\nu}\le\theta.
\end{equation}
The number of intervals depends only on $M,T,\nu$. Assume as a bootstrap that $E_u(I_m)+E_b(I_m)\le1$. Then the exponential factors in \eqref{eq:Eu-app}--\eqref{eq:Eb-app} are uniformly bounded. From \eqref{eq:Eb-app}, after absorbing $CA_u(I_m)E_b(I_m)$,
\begin{equation}
\label{eq:Eb-solved}
  E_b(I_m)
  \le C\big(e_m+\mathcal R_b(I_m)\big)
  +C\nu^{-1}A_b(I_m)E_u(I_m)
  +C\nu^{-1}E_u(I_m)E_b(I_m),
\end{equation}
where
\[
  e_m:=\norm[B^{s-1}_{\infty,1}]{w(t_m)}
       +\norm[B^s_{\infty,1}]{z(t_m)}.
\]
Insert \eqref{eq:Eb-solved} into the linear magnetic term in \eqref{eq:Eu-app}. The coefficient of $E_u(I_m)$ created in this way is bounded by
\[
  C\frac{|I_m|A_b(I_m)^2}{\nu}\le C\theta.
\]
All other linear terms are bounded by $C\theta(E_u+E_b)$. The remaining nonlinear terms are quadratic in $E_u+E_b$. Consequently,
\begin{equation}
\label{eq:combined-bootstrap}
\begin{aligned}
  E_u(I_m)+E_b(I_m)
  &\le C\big(e_m+\mathcal R_u(I_m)+\mathcal R_b(I_m)\big)\\
  &\quad+C\theta\big(E_u(I_m)+E_b(I_m)\big)
  +C\big(E_u(I_m)+E_b(I_m)\big)^2.
\end{aligned}
\end{equation}
Choose $\theta$ so that the linear term is absorbed. If the initial error and the total residual are below a sufficiently small threshold, a continuity argument improves the bootstrap and yields
\begin{equation}
\label{eq:interval-bound}
  E_u(I_m)+E_b(I_m)
  \le C\big(e_m+\mathcal R_u(I_m)+\mathcal R_b(I_m)\big).
\end{equation}
The endpoint norms at $t_{m+1}$ are controlled by the left-hand side. Iterating \eqref{eq:interval-bound} over the finite partition gives \eqref{eq:stability-bound}, with constants depending only on $M,T,\nu$. Choosing the total residual below the bootstrap threshold also prevents breakdown of the smooth local solution before $T$. Finally, $B^{s-1}_{\infty,1}\hookrightarrow L^\infty$ and $B^s_{\infty,1}\hookrightarrow L^\infty$ prove the last assertion of \cref{prop:stability}.

\section*{Declaration of competing interest}
The authors declare that they have no known competing financial interests or personal relationships that could have appeared to influence the work reported in this paper.

\section*{Funding}
This research did not receive any specific grant from funding agencies in the public, commercial, or not-for-profit sectors.

\section*{Data availability}
No data were used for the research described in this article.

\section*{Declaration of generative AI and AI-assisted technologies in the manuscript preparation process}
During the preparation of this work, the authors used OpenAI ChatGPT to assist with language editing, LaTeX organization, and consistency checks. After using this tool, the authors reviewed and edited the content as needed, independently verified the mathematical arguments, and take full responsibility for the content of the article.

\end{document}